\documentclass[12 pt]{amsproc}
\usepackage{graphicx}
\usepackage{amssymb}
\usepackage{epstopdf}
\usepackage{amsmath,amsthm,amscd,amssymb}
\usepackage{latexsym}
\usepackage[colorlinks,citecolor=red,pagebackref,hypertexnames=false]{hyperref}
\usepackage{geometry}                
\numberwithin{equation}{section}

\theoremstyle{plain}
\newtheorem{theorem}{Theorem}[section]
\newtheorem{lemma}[theorem]{Lemma}

\newtheorem{proposition}[theorem]{Proposition}

\theoremstyle{definition}
\newtheorem{definition}[theorem]{Definition}

\newtheorem{claim}[theorem]{Claim}

\newtheorem{case[theorem]}{Case}

\theoremstyle{remark}
\newtheorem{remark}[theorem]{Remark}

\numberwithin{equation}{section}

\def\be{\begin{equation}}
\def\ee{\end{equation}}

\def\od{\mathbb O(d)}
\def\odk{\mathbb O(d-k)}
\def\R{\Bbb R}

\def\bx{{\bf x}}
\def\by{{\bf y}}
\def\stabx{\text{Stab}(\bx)}
\def\staby{\text{Stab}(\by)}

\def\vt{\vec{t}\,}
\def\vv{\vec{v}\,}

\def\dkre{\Delta_{k}^{r}(E)}
\def\hd{\dim_{\mathcal H}}
\def\N{\Bbb N}

\def\({\left(}
\def\){\right)}
\def\[{\left[}
\def\]{\right]}
\def\<{\left\langle}
\def\>{\right\rangle}

\def\d{\partial}

\def\diam{\hbox{diam}}
\def\supp{\hbox{supp}}

\def\qed{\hfill\Box\smallskip}

\begin{document}

\title{Existence of similar point configurations  in thin subsets of $\R^d$}

\author{Allan Greenleaf, Alex Iosevich and Sevak Mkrtchyan}

\date{February 12, 2021}

\email{allan@math.rochester.edu}

\email{iosevich@math.rochester.edu}

\email{sevak.mkrtchyan@rochester.edu}

\address{Department of Mathematics, University of Rochester, Rochester, NY 14627}

\thanks{Revised, 12 February 2021. The first listed author was partially supported by NSF DMS-1362271 and -1906186. 
The second listed author was partially supported by NSA H98230-15-1-0319. 
The third listed author was partially supported by the Simons Foundation Collaboration 
Grant No. 422190.}

\maketitle

\begin{abstract} We prove the existence of similar and multi-similar point configurations (or simplexes) 
in sets of fractional Hausdorff dimension in Euclidean space. 
 Let $d \ge 2$ and $E\subset {\Bbb R}^d$ be a compact set.
For $k\ge 1$, define 
$$\Delta_k(E)=\left\{\left(|x^1-x^2|, \dots, |x^i-x^j|,\dots, |x^k-x^{k+1}|\right): 
\left\{x^i\right\}_{i=1}^{k+1}\subset E\right\} \subset {\Bbb R}^{k(k+1)/2},
$$ 
the {\it $(k+1)$-point configuration set} of $E$.
For $k\le d$, this is (up to permutations) the set of congruences of $(k+1)$-point configurations in $E$;
for $k>d$, it is the edge-length set of $(k+1)$-graphs whose vertices are in $E$.
Previous works by a number of authors have found values $s_{k,d}<d$ so that if the Hausdorff dimension of $E$  
satisfies $\hd(E)>s_{k,d}$,
then  $\Delta_k(E)$ has positive Lebesgue measure. 
In this paper we study more refined properties of $\Delta_k(E)$, 
namely the existence of  similar or multi--similar configurations.
For  $r\in\R,\, r>0$, let 
$$\Delta_{k}^{r}(E):=\left\{\vt\in \Delta_k\left(E\right): r\vt\in \Delta_k\left(E\right)\right\}\subset \Delta_k\left(E\right).$$
We show that if $\hd(E)>s_{k,d}$, for a natural measure $\nu_k$ on $\Delta_k(E)$,
one has $\nu_k\left(\Delta_{k}^{r}\left(E\right)\right)>\nolinebreak0$ all $r\in\R_+$.
Thus,   in $E$ there exist  many pairs of $(k+1)$-point configurations  
which are similar by the scaling factor $r$. 
 We extend this to show the existence of multi--similar configurations of any multiplicity. 
These results can be viewed as variants and extensions, for compact thin sets, of theorems 
of Furstenberg, Katznelson and Weiss \cite{FKW90}, Bourgain \cite{B86} and Ziegler \cite{Z06} 
for sets of positive density in $\R^d$.
\end{abstract} 

\maketitle


\section{Introduction}\label{sec intro}

\vskip.125in 

Furstenberg, Katznelson and Weiss \cite{FKW90} proved that if $A \subset {\Bbb R}^2$ has positive upper 
Lebesgue density and $A_{\delta}$ denotes its $\delta$-neighborhood, then, 
given vectors $u,v$ in ${\Bbb R}^2$, there exists $r_0$ such that,
 for all $r>r_0$  and any $\delta>0$,  there exists $\{x,y,z\} \subset A_{\delta}$ 
 forming a triangle 
 similar to $\{{\bf 0},u,v\}$, i.e.,
 congruent to $\{{\bf 0},ru,rv\}$ for some scaling factor $r>0$. 
Under the same assumptions, Bourgain \cite{B86} proved  in ${\Bbb R}^d$ that 
if $u^1, \dots, u^k \in {\Bbb R}^d$, $k \leq d$, there exists $r_0$ such that, for all $r>r_0$ 
and any $\delta>0$, there exists $\{x^1,x^2, \dots, x^{k+1}\} \subset A_{\delta}$ forming a 
simplex similar to $\{ {\bf 0}, u^1, \dots, u^k\}$ via scaling $r$. 
Perhaps the most general result of this approximate similarity type is due to Ziegler \cite{Z06}.
\vskip.125in

Bourgain also showed 
that if $k<d$ and the  simplex is non-degenerate, i.e., of positive $k$-dimensional volume, 
then the conclusion in \cite{B86}  holds for exact similarities ($\delta=0$):
for $r$ sufficiently large,
there exists $\{x^1,x^2, \dots, x^{k+1}\}\subset A$ similar to $\{ {\bf 0}, u^1, \dots, u^k\}$ via scaling $r$.
The purpose of this paper is to prove variants of such exact similarity results for compact sets $E$  of Hausdorff dimension 
$\hd(E)<d$, 
sometimes referred to as {\it thin sets}. 
The statements necessarily need to be different in this context than for sets
of positive upper 
Lebesgue density since, e.g.,  
in a compact $E$  the distances between points are  bounded above by $\diam(E)$. 
More fundamental restrictions are known:
e.g.,  there exist compact subsets of ${\Bbb R}^2$ of full Hausdorff dimension 
that do not contain vertices of any equilateral triangle (Falconer \cite{Falc80}). 
Nevertheless, we are able to prove that if $\hd(E)$ is above a threshold, 
then given \emph{any} $r>0$ there exist many pairs (in fact, a set of positive measure in an appropriate sense) 
of $(k+1)$-point configurations which are
similar via the scaling factor $r$ and whose vertices 
are, as in \cite{B86},  in $E$ itself and not just in a $\delta$-neighborhood of $E$. 
We also treat what we call \emph{multi-similarities}: multiple  configurations in 
$E$ which are jointly similar to each other via multiple scaling factors. 

\vskip.125in 

We now turn to the results of this paper; to state them, some definitions are needed.

\begin{definition} \label{kdistmeasuredef} 
Let   $d \ge 2$,  $E \subset {\Bbb R}^d$
be a compact set, and $\mu$  a Frostman measure on $E$. 
For $1\le k\le d$,
denote points of ${\Bbb R}^{k+1 \choose 2}={\Bbb R}^{k(k+1)/2}$ by $\vt:=(t^{ij})$,
and for $x^1,\dots, x^{k+1}\in\R^d$,
let $\vv_{k,d}(x^1, \dots, x^{k+1})\in \R^{k(k+1)/2}$ be the vector with 
entries $|x^i-x^j|$, $1 \leq i<j \leq k+1$, listed in the lexicographic order. 
\smallskip

(i) The {\it $k$-simplex set} or {\it $(k+1)$-point configuration set} of $E$ is
\begin{equation} \label{def kset} 
\Delta_k(E):=\left\{\vv_{k,d}(x^1, \dots, x^{k+1}): x^j \in E \right\}\subset {\Bbb R}^{k(k+1)/2}.
\end{equation}
Note that the $k$-simplex will necessarily be degenerate if $k>d$.

(ii) Define a measure $\nu_{k}$ on ${\Bbb R}^{k(k+1)/2}$, induced by a Frostman measure
\footnote{Recall that a compact  $E\subset\R^d$ always supports a Frostman measure; 
see \cite{Mat95} for  the definition of Frostman measures and their basic properties.}
 $\mu$ on $E$,
by the relation, for $f(\vt) \in C_0\left({\Bbb R}^{k(k+1)/2}\right)$,  
\begin{equation} \label{kmeasuredef} 
\int_{{\Bbb R}^{k(k+1)/2}} f(\vec{t})\,  d\nu_{k}(\vec{t})
:=\int \dots \int f\left(\vv_{k,d}\left(x^1, \dots, x^{k+1}\right)\right)\, d\mu(x^1) \cdots d\mu(x^{k+1}).
\end{equation}

\end{definition} 

\vskip.125in 

\begin{remark}
For $k\le d$, $\Delta_k(E)$ can be considered, modulo  the symmetric group $S_{k+1}$ acting on the $x^i$,
as the set of congruence classes of $(k+1)$-point configurations in $E$, 
or equivalently the set of (possibly degenerate) $k$-simplexes in $\R^d$ spanned by points of $E$.
The measure $\nu_k$ is  supported on $\Delta_k(E)$ and  has total mass at most $\mu(E)^{k+1}$. 
The action of the finite group $S_{k+1}$ will be  irrelevant for our results, 
which are expressed in terms of certain sets of configurations having positive $k(k+1)/2$-dimensional Lebesgue measure.
(The situation when $k>d$ is discussed below.)
\end{remark}

\vskip.125in 

The study of the Lebesgue measure of the distance set $\Delta_1(E)$ for thin sets was begun  in 1986 by Falconer \cite{Falc86}
and has led to many important results, comprising too large a literature to summarize here.
Most directly relevant,  the best  results currently known for the positivity of the measure of $\Delta_k(E),\, 1< k \le d$, 
are due to Erdo\u{g}an, Hart and Iosevich \cite{EHI13} and Greenleaf,
Iosevich, Liu and Palsson \cite{GILP13}. The former proved that the Lebesgue measure ${\mathcal L}^{k(k+1)/2}(\Delta_k(E))>0$ if 
$\hd(E)>\frac{d+k+1}{2}$,
while the latter obtained the threshold $\frac{dk+1}{k+1}$, 
improved to $\frac{8}{5}$ for $k=d=2$. 
We note that all of these results, except for \cite{GIOW18}, 
are proven by establishing that the measure $\nu_k$ defined by \eqref{kmeasuredef} has a density in  $L^2({\Bbb R}^{k(k+1)/2})$.
\vskip.125in

\begin{definition} \label{L2threshold} 
Let $d\ge 2,\, k\ge 1$. The {\it  $L^2$-threshold} for the $k$-simplex problem
(or  $(k+1)$-point configuration problem)
in $\R^d$  is
$$s_{k,d}:=\inf \left\{s: \hd(E)>s \implies \nu_k \hbox{ is absolutely continuous and} \int_{\Delta_k(E)} \nu^2_k(\vec{t})\, d\vec{t}<\infty \right\},$$ 
 where $E$ runs over all compact sets $E \subset {\Bbb R}^d$.
\end{definition} 

Thus, $s_{k,d}$ is less than or equal to the values mentioned in the paragraph above.

\vskip.125in

The case of $k>d$ needs to be treated somewhat differently, since in that range the set $\Delta_k(E)\subset\mathbb{R}^{k(k+1)/2}$ 
has lower dimension than $\mathbb{R}^{k(k+1)/2}$ and so cannot have positive Lebesgue measure, regardless of the Hausdorff dimension of $E$. 
This stems from the fact that, when $k>d$, specifying the $k(k+1)/2$ pairwise distances between $k+1$ 
points in $\mathbb{R}^d$ gives an over-determined system: 
knowing only some of the distances determines the rest. 
Thus,  although $\Delta_k(E)$ still makes sense, the setup has to be modified. 
In Chatzikonstantinou, Iosevich, Mkrtchyan and Pakianathan \cite{CIMP17}  it was shown that for $k>d$ the set of congruence 
classes of $(k+1)$-tuples of elements of $E$
can be naturally viewed as a subset of ${\Bbb R}^{d(k+1)-{d+1 \choose 2}}$; if $m:=d(k+1)-\binom{d+1}{2}$ appropriately chosen distances are 
specified, then  the other  distances are determined, up to finitely many possibilities. 
Let $P$ be such a collection of $m$ edges. 
In the terminology of \cite{CIMP17}, $P$ is a maximally independent 
(in $\mathbb{R}^d$) 
subset of the edges of the complete graph on $k+1$ vertices. 
Extend the definition of $\vv_{k,d}$ to the case $k>d$ by setting $\vv_{k,d}(x^1, \dots, x^{k+1})=\left(\left|x^i-x^j\right|\right)_{(i,j)\in P}\in {\Bbb R}^{m}$, where the 
entries in the range are ordered lexicographically. Using this, we can define a measure $\nu_k$ on $\mathbb{R}^m$ and 
a set $\Delta_k(E)\subset\mathbb{R}^m$ similarly to (\ref{kmeasuredef}) in Def. \ref{kdistmeasuredef}. Note that $\nu_k$ and $\Delta_k(E)$ will 
depend on the choice of $P$, but for our purposes this is irrelevant, so we will fix a particular $P$ once and for all.
\vskip.125in 

While $\vv_{k,d}(x^1, \dots, x^{k+1})$ doesn't determine the congruence class of $(x^1,\dots,x^{k+1})$ uniquely, 
it identifies it up to a finite number of 
possibilities. The number of these possibilities is bounded above by a constant $u_{d,k}$, depending only on $d$ and $k$. 
In this sense, congruence classes of $(k+1)$-tuples of elements of 
a compact set $E$ in ${\Bbb R}^d$ for $k>d \ge 2$ 
can be naturally viewed as a subset of 
$\R^m$. It was shown in  \cite{CIMP17} that if $k>d$ 
and the Hausdorff dimension of $E$ is greater than $d-\frac{1}{k+1}$, 
then the 
$m$-dimensional Lebesgue measure of the set of congruence 
classes of $(k+1)$-point configurations with endpoints in $E$ is positive,
and, as with most of the results for $k\le d$,  this was shown by first establishing that 
the measure $\nu_k$ defined by \eqref{kmeasuredef} has a density in   $L^2(\mathbb{R}^m)$.
\vskip.125in

With the theorems of \cite{FKW90,B86,Z06} in mind, obtaining more refined structural information 
about $\Delta_k(E)$ is of natural interest. 
Since $E$ is compact
in our setting, the questions need to reflect the fact, that all the pairwise distances are bounded. 
We will show that, if $\hd(E)>s_{k,d}$, 
then 
among the $k$-simplexes of $E$, all possible similarity scaling factors, $0<r<\infty$,   
occur, and each one does so with positive $\nu_k$-measure.
Furthermore, we show that what we call {\it multi-similarities} of arbitrarily high multiplicity occur as well. Thus, this holds for the values
of $\hd(E)$  in the settings of all of  the  previous results referred to above where ${\mathcal L}^{k(k+1)/2}(\Delta_k(E))>0$ 
has been obtained, with the possible  exception of \cite{GIOW18}.
\vskip.125in

To make this more precise,
for  $r\in\R_+:=(0,\infty)$ let 
\be\label{def drk}
\Delta_{k}^{r}(E):=\left\{\vt\in \Delta_k\left(E\right): r\vt\in \Delta_k\left(E\right)\right\}\subset \Delta_k\left(E\right),
\ee
the set of all $k$-simplexes $\vt$ in $E$ for which there is also a  simplex in $E$ similar to $\vt$ via the scaling factor $r$.
Interchanging the roles of the two simplexes in such a pair, $\{\vt,\, r\vt\}\subset\Delta_k(E)$, note for later use that 
\begin{equation}\label{eqn rinverse}
\Delta_{k}^{1/r}(E)=\Delta_{k}^{r}(E).
\end{equation}
\vskip.125in 

One can not only  look for similar pairs $\{\vt,r\vt\}\subset \Delta_k(E)$, but more generally for similarities of higher multiplicity.

\begin{definition}\label{def multi}
A collection $\{\vt,\, r_1\vt,\, \dots ,\, r_{n-1}\vt\}\subset \Delta_k(E)$, with $\{1,r_1,\dots, r_{n-1}\}$ pairwise distinct, is an {\it $n$-similarity of $k$-simplexes in $E$,} also referred to as a {\it multi-similarity of multiplicity $n$}.
\end{definition}

\vskip.125in 

Our main results are the following. All are obtained under the assumptions that $d \ge 2$, $k\ge1$,  $E\subset{\Bbb R}^d$ is compact, $\mu$ is a Frostman measure on $E$ and $\nu_k$ is the measure induced by $\mu$ as in Def. \ref{kdistmeasuredef}. 
The following three results establish (quantitatively) the existence in $E$ of many multi-similarities of multiplicities 2, 3 and $n\ge 3$, resp.

\begin{theorem} \label{main1} Let $d\ge 2$, $E\subset\R^d$ compact and $k\ge 1$. Suppose that $\hd(E)>s_{k,d}$, the $L^2$-threshold for the $k$-simplex problem from Def. \ref{L2threshold}. 
Then there is a  lower bound, uniform in $r,\, 0<r<\infty$,
$$\nu_k(\Delta_k^r(E))\ge C(k,E)>0.$$

\end{theorem}

\vskip.125in 

With the same notation and assumptions as in Thm. \ref{main1}, we also have:

\begin{theorem} \label{main2} Suppose that $\hd(E)>s_{k,d}$. Then there exist distinct $r_1,r_2>0$, with 
$$\nu_k\left(\Delta_k^{r_1}\left(E\right)\cap \Delta_k^{r_2}\left(E\right)\right)>0.$$
In fact, for any partition $\R_+=\coprod_{\alpha\in A} R_\alpha$ with 
each $R_\alpha\ne\emptyset$ and countable, there exist distinct $\alpha_1,\alpha_2\in A$
and $r_1\in R_{\alpha_1},\, r_2\in R_{\alpha_2}$, such that 
$\nu_k\left(\Delta_k^{r_1}\left(E\right)\cap \Delta_k^{r_2}\left(E\right)\right)>0$.
\end{theorem}

\vskip.125in 

\begin{theorem} \label{main3} Suppose that $\hd(E)>s_{k,d}$. Then for all $n\in\N$, there exists an $M=M(n,k,E)\in\N$ such that for any  distinct $r_1,\dots,r_{M}\in\R_+$, there
exist distinct $r_{i_1},\dots, r_{i_n}$ such that
$$\nu_k\left(\Delta_k^{r_{i_1}}\left(E\right)\cap \Delta_k^{r_{i_2}}\left(E\right)\cap \cdots \cap \Delta_k^{r_{i_n}}\left(E\right)\right)>0.$$
\end{theorem}

\vskip.125in 

\begin{remark} \label{remark main}
More explicitly, Thm. \ref{main1} says that, given any $r>0$, there exist $(k+1)$-point configurations $\{x^1, x^2, \dots, x^{k+1}\}$ and $\{y^1, y^2, \dots, y^{k+1}\}$ in $E$ which are similar via the scaling factor $r$, 
i.e., there exists a translation $\tau \in {\Bbb R}^d$ and a rotation $\theta \in O_d({\Bbb R})$ such that $y^j=r \theta (x^j+\tau),\, 1\le j\le k+1$. 
Thus, among the $(k+1)$-point configurations or $k$-simplexes in $E$,   in the terminology of Def. \ref{def multi},
there exist (many) 2-similarities.
\vskip.125in 

Furthermore,  Thm. \ref{main2} states that there exist 3-similarities in $E$, i.e.,   triples of $(k+1)$-point configurations, $\{x^j\},\, \{y^j\},\, \{z^j\}$ in $E$  
and scalings $r_1,\, r_2$ so that $y^j=r_1 \theta_1 (x^j+\tau_1),\, z^j=r_2\theta_2(x^j+\tau_2)$ for appropriate
$\theta_1,\theta_2\in O_d({\Bbb R})$ and $\tau_1,\tau_2\in\R^d$, and furthermore that $r_1,\, r_2$ can be arranged to lie in different subsets of a 
partition of $\R_+$ as stated. For example, decomposing $\R_+$ into the multiplicative cosets of $\Bbb Q_+$,
 there exist similarities of multiplicity  3 with $r_2/r_1$  irrational; similarly, replacing the positive rationals with the positive algebraic numbers, there 
 exist such with $r_2/r_1$ transcendental.
 \vskip.125in 

Finally, Thm. \ref{main3} shows that there exist multi-similar $(k+1)$-point configurations in $E$ of arbitrarily high multiplicity, and that the scaling 
factors can be chosen to come from an arbitrary set of distinct elements of $\R_+$, as long as that set has large enough cardinality relative to the 
desired similarity multiplicity.
\end{remark}

\begin{remark}
We note that the conclusions of Remark \ref{remark main} hold when $k>d$ as well. 
Denoting $\bx:=(x^1,\dots,x^{k+1})$, the fact that $\vv_{k,d}(\bx)=\vv_{k,d}(\by)$ does not imply that $\bx$ and $\by$ are congruent and hence $\vv_{k,d}(\bx)=r\vv_{k,d}(\by)$ 
does not imply $\bx$ and $\by$ are similar. 
However, the conclusions of Remark \ref{remark main} still hold as follows. Recall from the introduction that by results of Chatzikonstantinou, Iosevich, Mkrtchyan and Pakianathan \cite{CIMP17},
 $\vv_{k,d}(\bx)$ determines 
the congruence type of $\bx$ up to at most  a bounded number $u_{d,k}$ of choices. 
Using Thm. \ref{main3} with $n\cdot u_{d,k}$ instead of $n$ we see 
that there exist $\bx,\bx_{i_1},\dots,\bx_{i_{nu_{d,k}}}$ and $r_{i_1},\dots,r_{i_{nu_{d,k}}}$ 
such that $\vv_{k,d}(\bx),r_{i_1}\vv_{k,d}(\bx_1),\dots,r_{i_{nu_{d,k}}}\vv_{k,d}(\bx_{nu_{d,k}})$ are all congruent. 
It follows that $\bx,\bx_{i_1},\dots,\bx_{i_{nu_{d,k}}}$ all fall within at most $u_{d,k}$ congruence classes; thus, 
by the pigeon hole principle, at least $n+1$ of them must be in some congruence class. 
This argument applies to the conclusions of Remark \ref{remark main} for the other Theorems as well.
\end{remark}

\vskip.25in 

\section{Proofs of Theorems \ref{main2} and \ref{main3}}\label{sec proofs23}

We start by showing that Thms. \ref{main2}  and  \ref{main3} follow from Thm. \ref{main1} by  measure-theoretic arguments. To prove Thm. \ref{main2}, 
let $\R_+=\coprod_{\alpha\in A} R_\alpha$ be a partition of $\R_+$ into a (necessarily uncountable) collection of nonempty countable subsets. 
From the definition \eqref{def drk}, it follows that each $\dkre$ is $\nu_k$-measurable. Hence, if for each $\alpha\in A$, 
with slight abuse of notation we 
define the set
$$
\Delta_k^\alpha(E):=\bigcup_{r\in R_\alpha} \dkre,
$$
then, being a countable union of measurable sets, each $\Delta_k^\alpha(E)$ is $\nu_k$-measurable. 
Furthermore,  combining $R_\alpha\ne\emptyset$, the monotonicity of $\nu_k$ and  Thm. \ref{main1},
one sees that each $\nu_k(\Delta_k^\alpha(E))>0$.
However, $\nu_k(\Delta_k(E))\le \mu(E)^{k+1}<\infty$,
and no finite (or even $\sigma$-finite) measure space can be the pairwise disjoint union of an uncountable collection of measurable 
subsets of positive measure. Thus, there must exist $\alpha_1\ne \alpha_2$ 
such that $\Delta_k^{\alpha_1}(E)\cap \Delta_k^{\alpha_2}(E)\ne\emptyset$; 
it follows that there are $r_j\in R_{\alpha_j},\, j=1,2$, 
such that $\Delta_k^{r_1}(E)\cap \Delta_k^{r_2}(E)\ne\emptyset$. 
For the full claim of Thm. \ref{main2},
that  there exist distinct $\alpha_1,\alpha_2\in A$
and $r_1\in R_{\alpha_1},\, r_2\in R_{\alpha_2}$, such that 
$\nu_k\left(\Delta_k^{r_1}\left(E\right)\cap \Delta_k^{r_2}\left(E\right)\right)>0$,
first make a choice of one representative from each 
of the  $R_\alpha$, then choose an arbitrary countably infinite subset of these, and finally
apply Thm. \ref{main3}.

\vskip.125in

For the proof of Thm. \ref{main3}, we use the  uniform lower bound from Thm. \ref{main1},
$\nu_k(\dkre)\ge C(E,k)>0,\, \forall r\in\R_+$,
combined with  $\nu_k(\Delta_k(E))<\infty$.
 Thm. \ref{main3} then follows from the following measure-theoretic pigeon-hole principle, which 
 might be of independent interest and whose  proof is deferred to the Appendix, Sec. \ref{sec app}.

 \begin{lemma}\label{lemma pigeon}
 Let $\mathcal X=(X,\mathcal M,\sigma)$ be a finite measure space.
 For  $0<c<\sigma(X)$, let $\mathcal M_c=\{A\in\mathcal M: \sigma(A)\ge c\}$. 
 Then, for every $n\in\N$, there exists an $N=N(\mathcal X,c,n)\in\N$ such that for any collection $\{A_1,\dots,A_N\}\subset \mathcal M_c$ of cardinality $N$, there is a subcollection $\{A_{i_1},\dots, A_{i_n}\}$ of cardinality $n$ such that $\sigma(A_{i_1}\cap\cdots\cap A_{i_n})>0$ and hence 
 $A_{i_1}\cap\cdots\cap A_{i_n}\ne\emptyset$.
 \end{lemma}
 
 \vskip.125in

\section{Proof of Theorem \ref{main1}} \label{sec proof1}

\vskip.125in 

To keep the exposition simple, we first prove Thm. \ref{main1} in the case $k\le d$. In Sections \ref{sec proof1} and \ref{sec prop} we will assume $k\le d$. In the case $k>d$ the arguments are very similar. Since $\vv_{k,d}(x^1, \dots, x^{k+1})$ determines the congruence type of $(x^1, \dots, x^{k+1})$ up to at most $u_{d,k}$ choices, the constant $u_{k,d}$ will appear throughout the proof. However, since the results here are up to multiplicative constants, this doesn't play any essential role.

For $\epsilon>0$, define a smooth approximation of $\nu_k$ on $\R^{k(k+1)/2}$  by the density
\begin{equation} \label{kmeasureapprox} 
\nu^{\epsilon}_{k}(\vec{t})=\int \dots \int \prod_{1 \leq i<j \leq k+1} \sigma^{\epsilon}_{t_{ij}}(x^i-x^j) \prod_{l=1}^{k+1} d\mu(x^l), \end{equation} 
where $\sigma_t$ is the normalized surface measure on the sphere of radius $t$ in $\R^d$ and 
$\sigma_t^{\epsilon}(x):=\sigma_t*\rho_{\epsilon}(x)$, with 
$\rho \in C_0^{\infty}({\Bbb R^d})$, $\rho \ge 0$, $\supp(\rho)\subset\{|t|<1\}$, $\int \rho=1$ and $\rho_{\epsilon}(x)=\epsilon^{-d} \rho(\epsilon^{-1}x)$. 
Then  each $\nu^{\epsilon}_{k} \in C^\infty_0$ and $\nu^{\epsilon}_{k}\to \nu_k$ weak$^*$ as $\epsilon\to 0$.
Thus, 
$$ \nu_k(\Delta_k^r(E))=\lim_{\epsilon\to 0} \int_{\R^{k(k+1)/2}} \nu_k^{\epsilon}(r\vec{t})\, d\nu_k(\vec{t}).$$ 
By (\ref{kmeasuredef}), for $\epsilon$ fixed, 
$$\int_{\R^{k(k+1)/2}} \nu_k^{\epsilon}(r\vec{t})\, d\nu_k(\vec{t})= \int \nu_k^{\epsilon}\left(r(x^1-x^2), \dots, r(x^k-x^{k+1})\right)\, d\mu(x^1) \dots d\mu(x^{k+1}).$$ 
Using the definition in (\ref{kmeasureapprox}), we see that this  is 
\begin{equation} \label{pregilp} 
\approx \epsilon^{-{k+1 \choose 2}} 
\int \dots \int_{\left\{ \left|\,|x^i-x^j\right|-r\left|y^i-y^j|\, \right|<\epsilon;\, 1 \leq i<j \leq k+1\right\}} \, d\mu(x^1) \dots d\mu(x^{k+1})d\mu(y^1) \dots d\mu(y^{k+1}), \end{equation}
which we denote by $I_\epsilon$.
Here, and throughout, we write $X\lesssim Y$ (resp. $X \approx Y$) if there exist constants $0<c<C$, 
depending only on $k$, $E$ and the choice of $\rho$  (and thus implicitly on $d$), 
such that $X \lesssim C Y$ (resp., $cY \leq X \leq CY$). Also, we denote the $2(k+1)$-fold product measure in \eqref{pregilp}  and similar occurences by $\mu^{2(k+1)}$.

\vskip.125in

For each rotation $\theta \in O_d({\Bbb R})$, define a measure $\lambda_{r, \theta}$ on $\R^d$  by  
$$ \int f(z)\, d\lambda_{r, \theta}(z)=\int \int f(u-r\theta v)\, d\mu(u)d\mu(v),\, f\in C_0\left(\R^d\right).$$ 
This is the push-forward of $\mu\times\mu$ under the map $(u,v)\to u-r\theta v$,
has total mass  $||\lambda_{r,\theta}||=\mu(E)^2$, and is supported in $E-r\theta E$.
We show below  that, if $\hd(E)>s_{k,d}$,  for a.e. $\theta$, the measure $\lambda_{r,\theta}$ is absolutely continuous with a density in $L^{k+1}(\R^d)$,
which we  denote by $\lambda_{r,\theta}(\cdot)$.
Let $d\theta$ denote the Haar probability measure on $O_d(\R)$.

\begin{proposition} \label{groupactionbasicth} 
With the notation above, 
\begin{equation} \label{kgilpest} 
\liminf_{\epsilon \to 0} I_\epsilon \approx \int \int \left(\lambda_{r, \theta}\left(z\right)\right)^{k+1}\, dz\, d\theta. \end{equation} 
\end{proposition} 

\vskip.125in 

By definition, the quantity on the right hand side of (\ref{kgilpest}) is finite if $\hd(E)>s_{k,d}$, the $L^2$-threshold for the $k$-simplex problem. Prop. \ref{groupactionbasicth}  was proved in \cite{GILP13} in the case $r=1$; the proof in the general case is similar, but we supply it in the next section for the sake of completeness. 
\vskip.125in 

Continuing with the proof of Thm. \ref{main1},  by H\"older we have 
\begin{eqnarray}\label{eqn Holder}
\mu(E)^2 &=& \int\int \lambda_{r,\theta}(z)\cdot 1\, dz\, d\theta \nonumber \\
&\le& \left(\iint \left(\lambda_{r, \theta}\left(z\right)\right)^{k+1}\, dz\, d\theta\right)^{\frac1{k+1}}
\times \left(\iint _{\supp(\lambda_{r,\theta})\times O_d(\R)} 1^\frac{k+1}{k}\, dz\, d\theta\right)^{\frac{k}{k+1}}.
\end{eqnarray}
Since $\supp(\lambda_{r,\theta})$, being contained in $ E-r\theta E$, has Lebesgue measure $\lesssim (1+r^d) \mu(E)$,
we divide both sides of \eqref{eqn Holder} by the second factor on the right hand side and raise to the $k+1$ power to obtain 
$$\mu(E)^{k+1}(1+r^d)^{-(k+1)}\lesssim \iint \left(\lambda_{r, \theta}\left(z\right)\right)^{k+1}\, dz\, d\theta.$$
Combining this with Prop. \ref{groupactionbasicth}, 
we conclude that, for $\hd(E)>s_{k,d}$ and $0<r\le 1$,
  \begin{equation}
\liminf_{\epsilon \to 0} \int \nu_k^{\epsilon}(r \vec{t})\, d\nu_k(\vec{t})  \approx \int \int \left(\lambda_{r, \theta}\left(z\right)\right)^{k+1}\, dz\, d\theta  \gtrsim 1  .
\end{equation} 
It follows that $ \liminf_{\epsilon \to 0} \nu_k(\{\vec{t}: r \vec{t} \in \Delta_{k,\epsilon}(E)\}) \gtrsim 1,$
where $\Delta_{k,\epsilon}(E)$ is the $\epsilon$-neighborhood of $\Delta_k(E)$. 
Since the sets $\{\vec{t}: r \vec{t} \in \Delta_{k,\epsilon}(E)\}$ are nested as $\epsilon\searrow 0$, we conclude that, for $0<r\le 1$,  
\begin{equation} \label{eqn key}
\nu_k(\{\vec{t}: r \vec{t} \in \Delta_k(E)\})\gtrsim 1.
\end{equation}

However, by \eqref{eqn rinverse}, $ \nu_k\left(\left\{\vec{t}: r \vec{t} \in \Delta_k(E)\right\}\right)=  \nu_k\left(\left\{\vec{t}: r^{-1} \vec{t} \in \Delta_k(E)\right\}\right)$;
therefore,  \eqref{eqn key} holds for $1\le r<\infty$ as well.
This completes the proof of Thm. \ref{main1}, 
up to the verification of Prop. \ref{groupactionbasicth}.

\section{Proof of Proposition \ref{groupactionbasicth}}  \label{sec prop}

We will follow closely the argument in \cite[Sec. 2]{GILP13}.
It will be convenient to denote an ordered $(k+1)$-tuple $(x^1,\dots,x^{k+1})$ of elements of $\R^d$ by $\bx$.
If the corresponding set $\{x^1,\dots,x^{k+1}\}$ is a nondegenerate simplex (i.e., affinely independent), then
$$\pi(\bx):=span\{x^2-x^1,\dots, x^{k+1}-x^1\}$$
is a $k$-dimensional linear subspace of $\R^d$. 
 $\Delta(\bx)$ will denote the (unoriented) simplex  generated by $\{x^1,\dots,x^{k+1}\}$, i.e., the closed convex hull,  which is 
 contained in the affine plane $x^1+\pi(\bx)$. Both $\pi(\bx)$ and $\Delta(\bx)$ are independent of the order of the $x^j$. 
If $\{x^1,\dots,x^{k+1}\}$  is similar to $\{y^1,\dots, y^{k+1}\}$  by a scaling factor $r$,  
then, up to permutation of $y^1,\dots, y^{k+1}$, there exists a $\theta\in\od $ such that $x^j-x^1=r\theta(y^j-y^1),\, 2\le j\le k+1$, 
which is equivalent with $x^j-x^i=r\theta(y^j-y^i), 1\le i<j\le k+1$, 
and $\Delta(\bx)=(x^1-r\theta y^1)+ r\theta\Delta(\by)$. 
\vskip.125in

The group $\od$ acts on the Grassmanians $G(k,d)$ and $G(d-k,d)$ of $k$ (resp., $d-k$) 
dimensional linear subspaces of $\R^d$, and if $\bx$ is similar to $\by$, one has
$\pi(\bx)=\theta\pi(\by)$ and $\pi(\bx)^\perp=\theta\big(\pi(\by)^\perp\big)$.
The set of $\theta\in\od$ fixing $\pi(\bx)$ is a conjugate of $\odk\subset\od$, and we refer to this as the \emph{stabilizer} of $\bx$, denoted $\stabx$.

\vskip.125in

For $\bx,\, \by$ similar, let $\tilde{\theta}\in\od$  be such that it transforms $\by$ to $\bx$. I.e. we have $\pi(\bx)=\tilde{\theta}\pi(\by)$ and  $x^i-x^j=r\tilde{\theta}\omega(y^i-y^j)$ for all $\omega\in \staby$. 
For each $\by$, take a  cover of $\od / \staby$ by balls  of radius $\epsilon$ (with respect to some Riemannian metric) 
with finite overlap.  
Since the dimension of
 $\od/\staby$ is that of  $\od / \odk$, namely
$$\frac{d(d-1)}{2} - \frac{(d-k)(d-k-1)}{2} = kd - \frac{k(k+1)}{2},$$
one needs $N(\epsilon)\lesssim \epsilon^{-\left(kd - \frac{k(k+1)}{2}\right)}$ balls to cover it.
In these balls, choose  sample points, $\tilde{\theta}_m(\by),\, 1\le m\le N(\epsilon)$.
\vskip.125in

For a fixed $\by$ one sees that
\begin{eqnarray*}
& & \left\{\bx: \left||x^i-x^j|-r|y^i-y^j|\right|\leq \epsilon, \ 1 \leq i<j \leq k+1 \right\}\quad \\
& &\,\subseteq \bigcup\limits_{m=1}^{N(\epsilon)} \Big\{\bx: 
\left|(x^i-x^j)-r\tilde{\theta}_m(\by) \omega(y^i-y^j)\right|\lesssim \epsilon, 
\\ & & \qquad\qquad\qquad\qquad 
\,\forall\, 1 \leq i<j \leq k+1,\, \omega\in\staby \Big\}.
\end{eqnarray*}
 Thus,  the right hand side of (\ref{pregilp}) is bounded above by
\begin{eqnarray}\label{samplesum}
\epsilon^{- \frac{k(k+1)}{2}} \int \dots \int\sum\limits_{m=1}^{N(\epsilon)} \mu^{k+1} \Big\{(\bx) &:& 
\left|(x^i-x^j)-r\tilde{\theta}_m(\by) \omega(y^i-y^j)\right|\lesssim \epsilon, \nonumber \\
& & \quad  \forall\, 1\, \leq i<j \leq k+1,\, \omega\in\staby \Big\}d\mu(y^1) \dots d\mu(y^{k+1}).
\end{eqnarray}
When picking the $N(\epsilon)$ balls, if each point of $O(d)/Stab(y)$ is covered by at most $p=p(d)$ of the balls, then the quantity in \eqref{samplesum}
also becomes,  when multiplied by $1/p$,  a lower bound for (\ref{pregilp}). 
Thus, the right hand side of (\ref{pregilp}) is comparable to the quantity in (\ref{samplesum}), which can  be rewritten as
\begin{multline} \label{gettingthere}
\epsilon^{- kd} \int \dots \int \sum\limits_{m=1}^{N(\epsilon)} \epsilon^{kd- \frac{k(k+1)}{2}}\mu^{k+1} \Big\{\bx: \left|(x^i-r\tilde{\theta}_m(\by)  
\omega y^i)-(x^j-r\tilde{\theta}_m(\by) \omega y^j)\right|\lesssim \epsilon, \\ \forall\, 1 \leq i<j \leq k+1, \omega\in\staby \Big\}d\mu(y^1) \dots d\mu(y^{k+1}) .
\end{multline}

Now fix $\by$. Since this holds for any choice of  sample points $\tilde{\theta}_m(\by)$, 
we can  pick these points such that they minimize (up to a factor of 1/2, say) the quantity
\begin{multline*}
 \mu^{k+1} \Big\{\bx: \left|(x^i-r\tilde{\theta}_m(\by) \omega y^i)-(x^j-r\tilde{\theta}_m(\by) \omega y^j)\right|\leq \epsilon, \, 
 \forall\, 1 \leq i<j \leq k+1, \omega\in\staby \Big\} .
\end{multline*}

The $N(\epsilon)$ preimages, under the natural projection from $\od$, of the balls  used to cover $\od / \staby$ are $\epsilon$-tubular neighborhoods of the preimages of the sample points  $\tilde{\theta}_m(\by)$, which
we denote  $T_1^{\epsilon},\ldots,T_{N(\epsilon)}^{\epsilon}$;
these depend on $\by$, but we suppress  that in the notation.
Since $dim(\od / \staby)=kd - \frac{k(k+1)}{2}$, each  $T^\epsilon_m$ has measure $\sim \epsilon^{kd-\frac{k(k+1)}2}$ with respect to the Haar measure $d\theta$. 
Since the infimum over a set is less than or equal to the average over the set, it follows that
\begin{multline*}
\mu^{k+1} \Big\{\bx: \left|(x^i-r\tilde{\theta}_m(\by) \omega y^i)-(x^j-r\tilde{\theta}_m(\by) \omega y^j)\right|
\leq \epsilon, 
\\\qquad\qquad\qquad\qquad\qquad\qquad\qquad\qquad\qquad\qquad\qquad
\,\forall\, 1 \leq i<j \leq k+1,
\omega\in\staby \Big\} \\
\quad\approx \frac{1}{\epsilon^{kd - \frac{k(k+1)}{2}}} 
\int\limits_{T_m^{\epsilon}} \mu^{k+1} \left\{\bx: \left|(x^i-r\theta y^i)-(x^j-r\theta y^j)\right|\leq \epsilon, \ 1 \leq i<j \leq k+1 \right\} \, 
d\theta.
\end{multline*}
 Since the  collection $\{T_m^{\epsilon}\}$ have pointwise finite overlap 
(uniformly in $\epsilon$ and $\by$),
\linebreak integrating in $\by$ we see that the quantity in  (\ref{gettingthere}) is 
\begin{multline*}
\approx \epsilon^{-kd} \int \dots \int\sum\limits_{m=1}^{N(\epsilon)} \, 
\int\limits_{T_m^{\epsilon}} \mu^{k+1} \left\{\bx: 
\left|(x^i-r\theta y^i)-(x^j-r\theta y^j)\right|\leq \epsilon, \ 1 \leq i<j \leq k+1 \right\} \\d\theta\,  d\mu(y^1) \dots d\mu(y^{k+1}),
\end{multline*}
which is then
\begin{multline*}
\approx \epsilon^{-kd} \int \dots \int \int_{\od} \mu^{k+1} \left\{\bx: 
\left|(x^i-r\theta y^i)-(x^j-r\theta y^j)\right|\leq \epsilon, \ 1 \leq i<j \leq k+1 \right\} \\ d\theta d\mu(y^1) \dots d\mu(y^{k+1}).
\end{multline*}
Applying  Fubini-Tonelli we can rewrite this as
$$\approx \epsilon^{-kd} \int_{\od} \mu^{2(k+1)} \left\{(\bx,\by): 
\left|(x^i-r\theta y^i)-(x^j-r\theta y^j)\right|\leq \epsilon, \ 1 \leq i<j \leq k+1 \right\}d\theta,$$
and taking the lim inf, we obtain a quantity comparable to the expression (\ref{kgilpest}). 
This completes the proof of Proposition  \ref{groupactionbasicth}, and thus Theorems \ref{main1}, \ref{main2}, and \ref{main3}.

\section{Open question}\label{sec open}

The following is a natural question pertaining to the subject matter of Thm. \ref{main1}:

\vskip.125in 

In \cite{BIT16} it was shown that if $E$ is a compact subset of ${\Bbb R}^d$, of Hausdorff 
dimension greater than $\frac{d+1}{2}$, then there exists a non-empty open interval $I$ such that, 
for any $t \in I$, there exist $x^1, x^2, \dots, x^{k+1} \in E$ such that 
$|x^{j+1}-x^j|=t$, $1 \leq j \leq k$. In view of Thm. \ref{main1}, it seems reasonable to ask: given 
any $r>0$, do there exist $x,y,z \in E$ such that $|x-z|=r|x-y|$? This can be regarded as a pinned 
version of the case $k=1$ of Thm. \ref{main1}, in the sense that the endpoint $x$ is common to 
both segments whose length is being compared\footnote{Work of Krystal 
Taylor and two of the authors   now partially addresses this,
showing that if $\hd(E)>(2d+1)/3$, the set of such $r$ at least contains an interval \cite[Thm. 1.4]{GIT20}.}. 
Similar questions can be raised when $k>1$. 


\section{Appendix: A measure-theoretic Pigeon Hole Principle}\label{sec app}

 Unable to find Lemma \ref{lemma pigeon}  in the literature,
 and believing that it should be useful for other problems, we prove it here. 
Without loss of generality the  total measure  $\sigma(X)$ can be normalized to be equal to 1, 
so for the proof we restate the result as

\begin{lemma}\label{lemma pigeon rev}
 Let $\mathcal X=(X,\mathcal M,\sigma)$ be a probability space.
 For  $0<c<1$, let $\mathcal M_c=\{A\in\mathcal M: \sigma(A)\ge c\}$. 
 Then, for every $n\in\mathbb{N}$, there exists an $N=N(\mathcal X,c,n)\in\nolinebreak\mathbb{N}$ such that, for any collection $\{A_1,\dots,A_N\}\subset \mathcal M_c$ of cardinality at least $N$, there is a subcollection $\{A_{i_1},\dots, A_{i_n}\}$ of cardinality $n$ such that $\sigma(A_{i_1}\cap\cdots\cap A_{i_n})>0$ and hence 
 $A_{i_1}\cap\cdots\cap A_{i_n}\ne\emptyset$.
\end{lemma}

To start the proof, first we establish the following claim, which is a quantitative strengthening of the statement for $n=2$:
\begin{claim}\label{claim}
 Let $\mathcal X=(X,\mathcal M,\sigma)$ be a probability space. Then for any $0<c<1$ there exists $P_c\in\mathbb{N}$ such that for any $N>P_c$, if $\{A_1,\dots,A_N\}\subset\mathcal{M}_c$, then there exist distinct $i,j\leq N$ such that $\sigma(A_i\cap A_j)\geq 
 c^3/3$.
\end{claim}
\begin{proof}
Suppose not. Let $S\subset(0,1)$ be the set of all $c\in(0,1)$ such that the statement of the claim is false, and suppose $c\in S$. 
Then for every $ N\in\mathbb{N}$ there exists a subset $\{A_1,\dots,A_{2N}\}\subset\mathcal{M}_c$ such that $\sigma(A_i)\geq c$ for 
all $i$ but $\sigma(A_i\cap A_j)<c^3/3$ for all $i\ne j$. 
Consider the sets $A_{2i-1}\cup A_{2i}$, $i=1,\dots,N$. 
We have

\begin{equation*}
\sigma(A_{2i-1}\cup A_{2i})=\sigma(A_{2i-1})+\sigma(A_{2i})-\sigma(A_{2i-1}\cap A_{2i})
> c+c-\frac{c^3}3=2c-\frac{c^3}3.
\end{equation*}
Since $\sigma(X)=1\geq\sigma(A_{2i-1}\cup A_{2i})$, this implies $1> 2c-\frac{c^3}3$. 
In particular, since $c<1$, we have $c\lesssim 0.52<3/5$; hence $[3/5,1)\cap S=\emptyset$.
Moreover,
\begin{align*}
\sigma\big((A_{2i-1}&\cup A_{2i})\cap (A_{2j-1}\cup A_{2j})\big)
\\&=\sigma\big((A_{2i-1}\cap A_{2j})
\cup(A_{2i-1}\cap A_{2j-1})
\cup(A_{2i}\cap A_{2j})
\cup(A_{2i}\cap A_{2j-1})\big)
\\&\leq \sigma(A_{2i-1}\cap A_{2j})
+\sigma(A_{2i-1}\cap A_{2j-1})
+\sigma(A_{2i}\cap A_{2j})
+\sigma(A_{2i}\cap A_{2j-1})
\\&< 4\frac{c^3}3\leq \frac{(2c-c^3/3)^3}3\hbox{ since } 0<c<1.
\end{align*}
Thus, there exist $N$ sets, namely $A_1\cup A_2,\dots,A_{2N-1}\cup A_{2N}$, such that each has measure at least $f(c):=2c-\frac{c^3}3$ but all pairwise intersections have measure less than $\frac{f(c)^3}3$. 

Thus, we have shown that if $c\in S$ then  $f(c)\in S$ as well. 
However if $0<c<1$, then there exists $k\in\mathbb{N}$ such that $f^k(c)>3/5$ and is thus $\notin S$ 
(where $f^k$ denotes $f$ composed with itself $k$ times). It follows that $S$ must be empty.
$\qed$
\end{proof}

\vskip.125in 

We use Claim \ref{claim} as a building block for the proof of Lemma \ref{lemma pigeon rev}, which is by
 induction on $n$. 
 If $n=1$, then we can take $N=1$, since any $A_{i_1}\in \mathcal{M}_c$ satisfies the statement. 
 If $n=2$ then any $N\ge \lceil{ 1/c}\rceil$ suffices, 
 since there cannot be more than $1/c$ pairwise disjoint sets of measure $\geq c$ each; alternatively, one may simply invoke
 Claim \ref{claim}.
 
 \vskip.125in 

Now suppose that the conclusion of  Lemma \ref{lemma pigeon rev} holds for some $n$, $n\ge 2$. 
Set $N=2N(\mathcal{X},c^3/3,n)+P_c$, and suppose  $\{A_1,\dots,A_N\}\subset\mathcal{M}_c$ is a collection of cardinality $N$. 
Since $N>P_c$, by Claim \ref{claim}  there exist distinct $i,j\leq N$ such that $\sigma(A_i\cap A_j)>\frac{c^3}{3}$. Let $B_1=A_i\cap A_j$. 
Removing $A_i$ and $A_j$ from the collection we still have $N-2>P_c$ sets, so can find another pair whose intersection has 
measure at least $\frac{c^3}3$. Repeating this procedure $N(\mathcal{X},c^3/3,n)$ times, one finds sets 
$B_1,\dots,B_{N(\mathcal{X},c^3/3,n)}\in\mathcal{M}_{c^3/3}$. 
By the induction hypothesis there exist $0<i_1<i_2<\dots<i_n\leq N(\mathcal{X},c^3/3,n)$ 
such that $\sigma(B_{i_1}\cap\dots\cap B_{i_n})>0$. 
Since $B_{i_1}\cap\dots\cap B_{i_n}$ is the intersection of $2n$ distinct sets from the collection $\{A_1,\dots,A_N\}$, 
the intersection of any $n+1$ of those $2n$ will have positive measure, completing the induction step. 
$\qed$

\vskip.25in

\end{document}